\documentclass{preprint}

\Title{$L^1$-Uniqueness of Kolmogorov Operators Associated to 2D Stochastic Navier-Stokes Coriolis Equations with Space-Time White Noise}
\ShortTitle{Kolmogorov Operators Associated to 2D Stochastic Navier-Stokes Coriolis Equations}
\Author{Martin Sauer} 
\Address{Institut f\"ur Mathematik, Technische Universit\"at Berlin, Stra\ss e des 17. Juni 136, D-10623 Berlin, Germany. E-mail address: \texttt{sauer@math.tu-berlin.de}.}
\Abstract{We consider the Kolmogorov operator $K$ associated to a stochastic Navier-Stokes equation driven by space-time white noise on the two-dimensional torus with periodic boundary conditions and a rotating reference frame, introducing fictitious forces such as the Coriolis force. This equation then serves as a simple model for geophysical flows. We prove that the Gaussian measure induced by the enstrophy is infinitesimally invariant for $K$ on finitely based cylindrical test functions and moreover $K$ is $L^1$-unique w.\,r.\,t. the enstrophy measure for sufficiently large viscosity.}
\AMSsub{76D05, 60H15, 76B03, 76M35}

\newcommand{\tphi}{\tilde{\phi}}
\newcommand{\To}{{\mathbb T}}

\newcommand{\zstar}{\Z^2_\ast}
\newcommand{\zplus}{\Z^2_+}
\newcommand{\Kol}{K_{\sigma,\nu}}
\newcommand{\Res}{\overline{R}_{\sigma,\nu}^n(\lambda)}
\newcommand{\invM}{\mu_{\sigma,\nu}}
\newcommand{\dinvM}{\, \mathrm{d}\invM}
\newcommand{\Dk}{\partial_k}
\newcommand{\Dl}{\partial_l}
\newcommand{\fcb}{\mathcal{F}C_b}

\begin{document}
\section{Introduction and Main Result}
The Navier-Stokes equations in two space dimensions are particularly well-studied and the existence of a unique, global strong solution is well-known. Perturbations with a Gaussian noise are also covered, for example using the weak martingale or the variational approach, cf. \cite{FlaGatMartingale, LiuGeneralCoercive}. For an overview on randomly forced $2$D fluids we refer to \cite{KuksinLecture}. However, all these results need a smooth noise which does not include the case of so-called space-time white noise. Such a perturbation has some technical drawbacks but also a very reasonable legitimation. It has been observed in several articles \cite{AlbeverioGibbsMeasureOlder, AlbeverioGibbsMeasureOld, AlbeverioGibbsMeasure} that the periodic Euler flow on the torus $\To^2$, which is the inviscid limit of the $2$D-Navier-Stokes equations, has a family of infinitesimally invariant measures $\invM$, the so-called enstrophy measures. These are exactly the unique invariant measures of the Ornstein-Uhlenbeck 
processes corresponding to the purely linear problem, i.\,e. omitting the convection term, and is explicitly given as an infinite product measure. However, such a rough noise has technical drawbacks in terms of regularity issues making a pathwise interpretation  difficult. For example the Ornstein-Uhlenbeck process mentioned above takes values in a Sobolev space of negative order, hence merely distributions, since the convolution with the Stokes semigroup is not regularizing enough. The nonlinear problem is not expected to have more regularity, thus the main difficulty is the appropriate definition of the convection term for such distributions. In \cite{DPD2DNavierStokes}, Da Prato and Debussche prove the existence of a strong solution (in the probabilistic sense) with values in a certain Besov space of full measure $\invM$ for every initial condition within that space. Moreover, uniqueness is proven using an additional condition involving the stationary Ornstein-Uhlenbeck process. The problem with the convection term is tackled with a so-called renormalization technique. The, in some sense, unnatural notion of uniqueness is improved by Albeverio and Ferrario in \cite{AlbFerUniqueness2DNSE} to a pathwise uniqueness result in the same space, where existence holds.

In this article we are concerned with the associated Kolmogorov operator to these equations and its Cauchy problem in $L^1(\invM)$. This is related to the uniqueness of the corresponding Martingale problem, in particular a weaker formulation concerning stationary solutions. We use the concept of $L^1$-uniqueness, i.\,e. the closure of the Kolmogorov operator (with appropriate domain) generates a $C_0$-semigroup on $L^1(\invM)$. $L^1$-uniqueness for the stochastic Navier-Stokes equations perturbed by space-time white noise has been shown by Stannat in \cite{Stannat2DNSE} for large viscosity $\nu$. Similar, but weaker results have been obtained chronologically in \cite{FlandoliGozzi, AlbFerUniquenessGenerator, StannatRegularizedEuler, AlbFerUniquenessGenerator2}. The regularity issues from above translate into poor support properties of $\invM$ containing only distributions. This implies a poor convergence of the Galerkin approximations of the convection term, which is the major difficulty in this approach.

Despite the vast literature on $2$D fluids, such equations are not a very realistic setting. In most cases, they are used as an example of approximations for fluid flows, where the vertical length scale is negligible compared to the horizontal ones. Such applications often appear in the studies of atmospheric or oceanic flows. On these huge length scales, the rotation of the earth cannot be neglected and fictitious forces appear in the equations. The fictitious forces concerning rotation are the centrifugal force and the Coriolis force. We incorporate these forces to obtain a toy model for geophysical flows. One of the intriguing observations is that the additional forces still have the same invariants, i.\,e. energy and enstrophy, thus supposedly keep the enstrophy measure as an invariant measure. We consider the following equations for the velocity field $u$ and hydrodynamic pressure $\pi$ on the two-dimensional torus $\To^2 \df (0,2\pi)^2$ with periodic boundary conditions.
\begin{empheq}[right=\quad\empheqrbrace]{equation}\label{eq:2DSNSCE}
\begin{aligned}
\partial_t u &= \nu \Delta u - (u \cdot \nabla)u- l e_3 \times u - \nabla \pi  + \eta &\text{in }[0,\infty) \times \To^2,\\
\Div u &= 0 &\text{in }[0,\infty) \times \To^2,\\
u(0)&= u_0  &\text{in } \To^2,
\end{aligned}
\end{empheq}
where $\eta$ is the so-called space-time white noise. The centrifugal force is of gradient type and can be hidden in the pressure $\pi$, whereas the Coriolis force is modeled in the so-called $\beta$-plane model, for a motivation we refer to \cite{Pedlosky}. In this model $l = \omega + \beta \xi_2$ with $\omega, \beta > 0$ denoting the angular velocity and its fluctuation around the equatorial line. Here, some part of the earth's surface is approximated by a rectangle and $\xi_2$ denotes the longitudinal component. We set $e_3 \times u = u^\bot$ where $u^\bot = (-u_2, u_1)^T$. As usual we consider this equation in the function space
\[
\textstyle H \df \Big\{ u \in L^2(\To^2;\R^2): \Div u = 0, \int u(\xi)\dxi = 0, u\cdot n \text{ is periodic}\Big\},
\]
where $n$ denotes the outward normal. The abstract evolution equation on $H$ is obtained after applying the (orthogonal) Helmholtz projection $\mathcal{P}: L^2(\To^2;\R^2) \to H$. This equation is given by
\begin{empheq}[right=\quad\empheqrbrace]{equation}\label{eq:SNSE-SDE}
\begin{aligned}
\mathrm{d} u(t) &= \Bigl(\nu A u(t) - B\big(u(t)\big) - C\big(u(t)\big)\Bigr)\dt + \sigma \dwt\\
u(0)&= u_0 \in H,
\end{aligned}
\end{empheq}
where $A \df \mathcal{P}\Delta$ is the Stokes operator, $B(u) \df B(u,u) \df \mathcal{P}(u \cdot \nabla)u$ and $C(u) \df \mathcal{P} l u^\bot$. The noise is represented by a cylindrical Wiener process $W(t)$ on $H$. We then consider the Kolmogorov operator associated to \eqref{eq:2DSNSCE} defined by
\begin{equation}\label{def:Kolmogorov2DSNSCE}
\big(\Kol \phi\big)(u) = \frac{\sigma^2}{2} \tr \big(D^2 \phi(u)\big) + \scp{\nu A u - B(u) - C(u)}{D\phi(u)}
\end{equation}
for $\phi \in \mathcal{F}C^2_b$, the space of all cylindrical functions on $H$. In detail
\[
\fcb^m := \Big\{ \phi (u) = \tphi(u_{k_1}, \dots, u_{k_n}) : n\in \N, \tphi \in C_b^m(\R^n), u_{k_i} = \scp{u}{e_{k_i}}\Big\}.
\]
In particular, we are interested in the the well-posedness of its Cauchy problem in $L^1(\invM)$. These enstrophy measures are explicitly known and given by
\begin{equation}
\invM \df \mathcal{N} \Big(0, \tfrac{\sigma^2}{2\nu} A^{-1}\Big),
\end{equation}
i.\,e. the invariant measures of the linear problem with drift $\nu A$ and diffusion $\sigma I$. Due to the invariance of the enstrophy for both vector fields $B$ and $C$, these measures are indeed  infinitesimally invariant for $\Kol$, in formula $\int \Kol \phi  \dinvM = 0$ for all $\phi \in \fcb^2$, and thus a reasonable candidate for a reference measure, see Section 3 for more details. In order to state the assumptions for our main result we need the following notation. Let $S(s) := \sum_{k \in \zstar} \abs{k}^{-2 s}$ denote the value of the convergent infinite series for $s > 1$.
\begin{assum}\label{assum:ViscosityCondition}
Assume that $\sigma, \nu > 0$ satisfy $\nu^3 > 40 S(2) \pi^{-2} \sigma^2$.\end{assum}
With the assumption above, the main result of this article is stated as follows.
\begin{thm}\label{thm:UniquenessCoriolis}
Let $\invM$ be the Gaussian measure related to the enstrophy, then $(\Kol, \fcb^2)$ is dissipative in $L^1(\invM)$, hence closable. Now suppose Assumption \ref{assum:ViscosityCondition} holds. Then the operator $(\Kol, \fcb^2)$ is $L^1$-unique. In particular, the closure $(\overline{K}_{\sigma,\nu}, D(\overline{K}_{\sigma,\nu}))$ generates a $C_0$-semigroup of contractions $P_t$ in $L^1(\invM)$ and $\invM$ is invariant for $P_t$. 
\end{thm}
Such a result has some implications concerning uniqueness of the associated martingale problem. In particular, the semigroup $P_t$ is Markovian and yields the transition probabilities of a stationary martingale solution of \eqref{eq:SNSE-SDE}. We refer to the monograph \cite{Eberle} for a detailed discussion on this subject. Furthermore, we can obtain the following corollary to Theorem \ref{thm:UniquenessCoriolis} for the system without rotation.
\begin{cor}
Set $\omega = \beta = 0$, i.\,e. the reference frame is fixed. Then, under Assumption \ref{assum:ViscosityCondition}, the operator $(\Kol, \fcb^2)$ is $L^1$-unique.
\end{cor}
Although we are able to extend the result of \cite{Stannat2DNSE} regarding a smaller lower bound for $\nu$, the limit for small viscosity parameter still remains a challenge. This lower bound for $\nu$ is due to the techniques used in the proofs that, more or less, absorb the nonlinear contributions of the generator $\Kol$ by the Stokes part. Note that in the pathwise formulation \cite{AlbFerUniqueness2DNSE, DPD2DNavierStokes} this assumption is not needed, thus it is somehow artificial.

The proof of Theorem \ref{thm:UniquenessCoriolis} is contained in the following sections. At first, we derive a spectral representation of \eqref{eq:SNSE-SDE} which together with the product structure of $\invM$ favors the use of finite dimensional spectral Galerkin approximations of $\Kol$. Sharp convergence results for the approximated vector fields $B$ and $C$ are contained in Lemmas \ref{lem:energyB} and \ref{lem:energyC} and in particular the definition of $B(u)$ and $C(u)$ via an $L^2$-limit for distributions $u$ from the support of $\invM$. The integration by parts formula in Lemma \ref{lem:PartialIntegrationGaussian} is essential for the main ingredient of proof, the a priori gradient estimate for the solution of the finite dimensional resolvent problem for $\Kol$ in Proposition \ref{prop:W1+sEstimate}. We seize this idea due to Stannat in \cite{Stannat2DNSE}, however suitable modifications for both the convection and Coriolis term are necessary. Note that this a priori estimate is not uniform in the 
approximation but introduces some logarithmic growth that is sufficiently small. With this approach we are able to weaken the smallness condition on the viscosity $\nu$ to some extent. 
\section{A Spectral Representation}
In the following, we expand the vector field $u$ into its Fourier series. We use the complete orthonormal system of $H$ given by
\[
e_k(\xi) \df \frac{1}{\sqrt{2}\pi} \frac{k^\bot}{\abs{k}} \phi_k(\xi) \quad\text{with}\quad \phi_k(\xi) \df
\begin{cases}
\sin(k \xi), \quad k \in \Z^2_+,\\
\cos(k \xi), \quad k \in -\Z^2_+,\\
\end{cases}
\]
where $\Z^2_+ \df \{ k \in \Z^2: k_1 > 0 \text{ or } (k_1 =0 \text{ and } k_2>0)\}$. Furthermore, let $\zstar \df \Z \setminus \{0\} = \Z^2_+ \cup -\Z^2_+$. It is an easy task to verify that $(e_k)\subset H$ and that these are eigenvectors of the Stokes operator $A$ with corresponding eigenvalues $-\abs{k}^2$. In addition to that, we can write the cylindrical Wiener process as a formal sum $W(t) = \sum_{k \in \zstar} \beta_k(t) e_k$ with a family $\beta_k$ of independent real-valued Brownian motions.

Expanding \eqref{eq:2DSNSCE} w.\,r.\,t. the orthonormal system $(e_k)$ yields a spectral representation of \eqref{eq:SNSE-SDE}, which is
\begin{equation}\label{eq:specNSCE}
\mathrm{d} u_k(t) = \Big(-\nu \abs{k}^2 u_k(t) - B_k\big(u(t)\big)-C_k\big(u(t)\big) \Big)\dt + \sigma \,\mathrm{d}\beta_k(t), \quad k \in \zstar.
\end{equation}
In a similar fashion the associated Kolmogorov operator reads as
\[
\big(\Kol \phi\big)(u) = \frac{\sigma^2}{2} \sum_{k \in \zstar} \Dk^2 \phi(u) - 2 \bigl(\nu \abs{k}^2 u_k + B_k(u) + C_k(u) \bigr) \Dk \phi(u).
\]
Here and in the following, we denote by $\Dk$ the derivative w.\,r.\,t. the $u_k$ variable. It remains to identity the Fourier coefficients of the convection and the Coriolis term, compare \cite{FlandoliGozzi} for similar calculations concerning the convection term. One easily verifies that 
\[
\scp{B(e_l, e_m)}{e_k} = - \scp{B(e_l,e_k)}{e_m} = \frac{\sqrt{2}}{4 \pi} \frac{(k^\bot \cdot l) (k \cdot m)}{\abs{k} \abs{l} \abs{m}} \frac{1}{\pi^2} \int_{\To^2} \bigl( \phi_{-k} \phi_l \phi_m \bigr) (\xi) \dxi \fd \beta^k_{l,m},
\]
hence
\begin{equation}\label{eq:fourierB}
B_k(u) \df \scp{B(u)}{e_k} = \sum_{l,m \in \zstar} u_l u_m \scp{B(e_l, e_m)}{e_k} = \sum_{l,m \in \zstar} \beta^k_{l,m} u_l u_m
\end{equation}
The integral in the definition of $\beta^k_{l,m}$ essentially yields some Kronecker deltas, $\delta_k = 1$ if $k=0$ or $\delta_k = 0$ otherwise. In detail
\[
\delta_{k,l,m} \df \frac{1}{\pi^2} \int_{\To^2} \bigl( \phi_{-k} \phi_l \phi_m \bigr) (\xi) \dxi = \begin{cases}
\delta_{k-l-m} - \delta_{k-l+m} - \delta_{k+l-m}, & k,l,m \in \zplus,\\
\delta_{k+l-m} + \delta_{k-l+m} + \delta_{k+l+m}, & k \in \zplus, l,m \in - \zplus,\\
\delta_{k-l-m} - \delta_{k-l+m} - \delta_{k+l+m}, & k,l \in -\zplus, m \in \zplus,\\
\delta_{k-l-m} - \delta_{k+l-m} - \delta_{k+l+m}, & k,m \in -\zplus, l \in \zplus,\\
0 &\text{otherwise.}
\end{cases}
\]
By similar calculations we can obtain the formula for the Coriolis forcing term. Note that only the second summand remains after applying the Helmholtz projection since $\mathcal{P} u^\bot = 0$. Analogue to the above, we see that
\[
C_k(u) \df \scp{C(u)}{e_k} = \beta \sum_{l \in \zstar} u_l \frac{k^\bot l}{\abs{k}\abs{l}} \frac{1}{2\pi^2}\int_{\To^2} \xi_2 \phi_k(\xi) \phi_l(\xi) \dxi
\]
and with straightforward calculations conclude that
\begin{equation}\label{eq:fourierC}
C_k(u) = -\beta \sum_{l \in \zstar} u_l \frac{k^\bot l}{\abs{k}\abs{l}} \frac{1}{k_2+l_2}\delta_{k_1+l_1} (1 - \delta_{k_2+l_2}) = -\beta \sum_{l \in \zstar} \gamma^k_l u_l.
\end{equation}

Furthermore, let us introduce some notation on the function spaces used in this article. With the complete orthonormal system $(e_k)$ the periodic, divergence free Sobolev spaces following \cite{BerghLoefstroem} can be identified with
\[
H^s := \Biggl\{ u \in \R^{\zstar}: \snorm{u}{s}^2 := \sum_{k \in \zstar} \abs{k}^{2 s} u_k^2 < \infty \Biggr\}, \quad s \in \R, \quad \text{ with }H^0 = H.
\]
Recall the complex interpolation of these Sobolev spaces which states that for $s_0 < s < s_1$ and $u \in H^{s_1}$ it holds that
\begin{equation}\label{eq:Interpolation}
\norm{u}_s \leq \norm{u}_{s_0}^{\frac{s_1-s}{s_1-s_0}} \norm{u}_{s_1}^{\frac{s-s_0}{s_1-s_0}}.
\end{equation}
\section{The Gaussian Invariant Measure Given by the Enstrophy}
In the coordinates of $(e_k)$ the measure $\invM$ is simply an infinite product of centered Gaussian measures on $\R$, i.\,e.
\[
\invM = \mathcal{N} \Bigl(0, \tfrac{\sigma^2}{2 \nu} A^{-1} \Bigr) = \bigotimes_{k \in \zstar} \mathcal{N} \Bigl(0, \tfrac{\sigma^2}{2 \nu \abs{k}^2} \Bigr).
\]
This measure is usually called the enstrophy measure, because the enstrophy associated to the vector field $u$ appears in the exponent of the heuristic density of $\invM$. It is well-known that $H$ does not have full measure w.\,r.\,t. $\invM$ and one even has $\invM(H) = 0$, cf. \cite{AlbeverioGibbsMeasureOlder}. This is due to the elementary calculation $\int u_k^2 \dinvM (u) = \frac{\sigma^2}{2 \nu \abs{k}^2}$, hence for $S_n(s) = \frac{2\nu}{\sigma^2} \sum_{k=1}^n \abs{k}^{-2s} u_k^2$ follows
\[
\lim_{n \to \infty} \int S_n(s) \dinvM (u) = \lim_{n \to \infty} \sum_{k=1}^n \abs{k}^{-2-2s} < \infty
\]
if and only if $s > 0$. This implies $\invM(H^{-s}) = 1$ if and only if $s>0$. The support of $\invM$ also contains all Sobolev spaces of negative order with any integrability parameter $1 \leq p < \infty$ and the Besov spaces $B^{-s}_{pq}$ for all $s >0$, $2 \leq p \leq q < \infty$, see \cite{AlbFerUniqueness2DNSE} for detailed computations.

As mentioned before, it has been shown that this measure is infinitesimally invariant for the Euler flow, see \cite{AlbeverioGibbsMeasure}, and also invariant for the Ornstein-Uhlenbeck process
\[
\mathrm{d} u(t) = \nu A u(t) \dt + \sigma \dwt,
\]
see for example \cite[Theorem 6.2.1]{dpzErgo}. In the following, we want to prove that $\invM$ is in fact infinitesimally invariant for $(\Kol, \fcb^2)$. This is mostly due to the two invariants $\scp{B(u)+C(u)}{u} = 0$ and $\scp{B(u) + C(u)}{Au} = 0$ for all smooth $u \in H$, see for example \cite{Titi}. This implies, at least for $u$ with $u_k \neq 0$ only for a finite number of $k \in \zstar$, that
\begin{equation}\label{eq:InvarianceBC}
\sum_{k \in \zstar} \abs{k}^{2 i} (B_k (u) + C_k(u)) u_k = 0, \quad i=0,1.
\end{equation}
Now fix a function $\phi \in \fcb^2$, hence there exists $k_1, \dots, k_n$, $n \in \N$ such that $\phi$ has an admissible representative $\tphi$ in $C_b^2(H_n)$ and for $\tilde{u} = (u_{k_1}, \dots, u_{k_n})$ follows
\begin{equation}\label{eq:InvarianceBCproof}
\int \Kol \phi (u) \dinvM(u) = \sum_{i=1}^n \int \bigl( B_{k_i}(u) + C_{k_i}(u) \bigr) \partial_{k_i} \tphi (\tilde{u}) \dinvM(u).
\end{equation}
On the right hand side we do an integration by parts w.\,r.\,t. the Gaussian density, more precisely we have the following lemma.
\begin{lem}\label{lem:PartialIntegrationGaussian}
Let $\phi \in \fcb^1$. Then
\[
\int \Dk \phi (u) \dinvM (u) = \frac{2 \nu}{\sigma^2} \abs{k}^2 \int u_k \phi (u) \dinvM (u),
\]
in particular
\[
\int \Dk \phi (u) (B_k (u) + C_k(u)) \dinvM (u) = \frac{2 \nu}{\sigma^2} \abs{k}^2 \int u_k \phi (u) (B_k(u) + C_k(u)) \dinvM (u).
\]
\end{lem}
\begin{proof}
We can use the product structure of the measure $\invM$ and obtain
\begin{align*}
\int \Dk \phi (u) \dinvM(u) &= \int \Dk \tphi (\tilde{u}) \, \mathrm{d} \Bigl( \bigotimes_{i=1}^n \mathcal{N} \bigl(0, \tfrac{\sigma^2}{2\nu \abs{k_i}^2}\bigr) \Bigr).\\
\intertext{If there exists $i^\ast$ with $k_{i^\ast} = k$ we do an integration by parts in this coordinate.}
&= \int \int_{-\infty}^\infty \Dk \tphi (\tilde{u}) e^{- \frac{\nu \abs{k}^2 u_k^2}{\sigma^2}} \, \mathrm{d}u_k \, \mathrm{d} \Bigl( \bigotimes_{i \neq i^\ast} \mathcal{N} \bigl(0, \tfrac{\sigma^2}{2\nu \abs{k_i}^2}\bigr) \Bigr)\\
&= - \int \int_{-\infty}^\infty \tphi (\tilde{u}) \Bigl( -\frac{2 \nu \abs{k}^2}{\sigma^2} u_k \Bigr) e^{- \frac{\nu \abs{k}^2 u_k^2}{\sigma^2}} \, \mathrm{d}u_k \, \mathrm{d} \Bigl( \bigotimes_{i \neq i^\ast} \mathcal{N} \bigl(0, \tfrac{\sigma^2}{2\nu \abs{k_i}^2}\bigr) \Bigr)\\
&= \frac{2 \nu \abs{k}^2}{\sigma^2} \int u_k \phi (u) \dinvM(u).
\end{align*}
Equations \eqref{eq:fourierB} and \eqref{eq:fourierC} imply that $\Dk B_k = 0$, $\Dk C_k = 0$ for all $k \in \zstar$, hence the second assertion follows easily if we replace $\phi$ by $\phi (B_k + C_k)$.
\end{proof}
Going back to \eqref{eq:InvarianceBCproof}, we can apply this integration by parts formula and obtain
\[
\int \Kol \phi (u) \dinvM(u) = \sum_{i=1}^n \frac{2 \nu \abs{k_i}^2}{\sigma^2} \int u_{k_i} \bigl( B_{k_i}(u) + C_{k_i}(u) \bigr) \tphi (\tilde{u}) \dinvM(u).
\]
Furthermore, we can use the invariance \eqref{eq:InvarianceBC} for $i=1$, which reads as $\sum_{i=1}^n u_{k_i} ( B_{k_i}(u) + C_{k_i}(u) ) = 0$ pointwise for all $u$ with $u_k = 0$ if $k \notin \{ k_1, \dots, k_n \}$. We conclude that
\begin{equation}\label{eq:InvarianceKol}
\int \Kol \phi(u) \dinvM (u) = 0 \quad \text{for all } \phi \in \fcb^2.
\end{equation}

Since the invariant measure is a product measure, it is reasonable to use the usual finite dimensional spectral Galerkin approximations in order to obtain an approximating equation for \eqref{eq:SNSE-SDE}.  Define $I_n := \{ k \in \zstar: \abs{k} \leq n \}$ and let $\iota_n$, $\pi_n$ be the canonical embedding and projection from and onto the subspace $H_n := \spann \{ e_k: k \in I_n \}$, respectively. Associated to $\iota_n$ and $\pi_n$ define $B^n(u) \df \sum_{k \in I_n} B_k^n(u)e_k$ and $C^n(u) \df \sum_{k \in I_n} C_k^n(u) e_k$ with $B_k^n(u) \df \sum_{l,m \in I_n} \beta^k_{l,m} u_l u_m$ and $C_k^n(u) \df -\beta \sum_{l \in I_n} \gamma^k_l u_l$, respectively. The approximating Kolmogorov operator $\Kol^n$ is defined in the canonical way by replacing all parts by the approximations
\[
\Kol^n \phi(u) = \frac{\sigma^2}{2} \sum_{k \in I_n} \Dk^2 \phi(u) - 2 \bigl(\nu \abs{k}^2 u_k + B^n_k(u) + C^n_k(u) \bigr) \Dk \phi(u).
\]
As a suitable domain we consider $C_b^2(H_n)$. It is clear that $\invM^n := \invM \circ \pi_n^{-1}$ is infinitesimally invariant for $(\Kol^n, C_b^2(H_n))$.

For the proof of Theorem \ref{thm:UniquenessCoriolis} it is essential in which sense $B^n$ and $C^n$ converge to $B$ and $C$. In the next two lemmas we obtain convergence in $L^2(\invM; H^{-s})$ for $s>1$ and $s>0$, respectively. As a byproduct, this allows to define the unique measurable extensions of the vector fields $B$ and $C$ to $L^2(\invM)$, i.\,e. an extension for distributions $u \in H^{-s}$ given any $s > 0$. Moreover, $L^2(\invM)$-convergence implies $\invM$-a.\,s. convergence along some subsequence, hence the limits $B(u)$ and $C(u)$ are in fact elements of $H^{-s}$ for $s>1$ and $s>0$, respectively.
\begin{lem}\label{lem:energyB}
Let $\sigma, \nu > 0$ be arbitrary. Then, $\snorm{B(u)}{-s} \in L^2(\invM)$ if and only if $s >1$. In particular, for all $0<\eps < \min \{s-1, 1\}$ it holds that
\[
\int \snorm{\pi_n (B - B^n)(u)}{-s}^2 \dinvM(u) \leq c \log (n) n^{-2\eps} \xrightarrow{n \to \infty} 0
\]
with a constant $c$ uniform in $n$. Moreover, $B(u)$ is an element of $H^{-s}$ for $\invM$-a.\,e. u.
\end{lem}
\begin{proof}
The first and last part of the statement have already been considered in the literature, see e.\,g. \cite[Proposition 3.2]{AlbFerUniqueness2DNSE}. A crucial part in the proof is the dependence of $\int \abs{B_k(u)}^2 \dinvM(u)$ on the index $k$. On can show that it is of order $\log (\abs{k})$. Essential for us however, is the explicit convergence rate of the approximations as $n \to \infty$. Note that
\[
\snorm{\pi_n (B - B^n)(u)}{-s}^2 = \sum_{k \in I_n} \abs{k}^{-2s} \abs{B_k(u) - B_k^n(u)}^2,
\]
thus we have to consider the difference of the $k$th Fourier coefficients in $L^2(\invM)$ for $k \in I_n$. Straightforward calculations yield
\begin{align*}
&\int \abs{B_k(u) - B_k^n(u)}^2 \dinvM(u) =  \sum_{\{l,m \in I_n\}^C} \bigl((\beta^k_{l,m})^2 + \beta^k_{l,m} \beta^k_{m,l} \bigr) \int u_l^2 u_m^2 \dinvM(u)\\
&\quad= \frac{\sigma^4}{\nu^2} \sum_{\{l,m \in I_n\}^C} \bigl((\beta^k_{l,m})^2 + \beta^k_{l,m} \beta^k_{m,l} \bigr) \abs{l}^{-2} \abs{m}^{-2} \leq c \abs{k}^2 \sum_{l \in \zstar, \abs{l}>n} \frac{1}{\abs{l}^2 \abs{k-l}^2}, 
\end{align*}
with a uniform constant $c$ independent of $k$ and $\{l,m \in I_n\}^C \df \{ l,m \in \zstar \setminus I_n\} \cup \{l \in I_n, m\in \zstar \setminus I_n\} \cup \{ l \in \zstar \setminus I_n, m \in I_n \}$. We can bound this sum by an integral and  make explicit calculations. At first, let $\abs{l} > 2n > 2\abs{k}$, then this part of the sum is bounded up to a uniform constant by
\[
\int_{ \{y \in \R^2: \abs{y} > 2n \}} \frac{1}{\abs{y}^2 \abs{k-y}^2}\dy \leq 8 \pi \int_{2n}^\infty \frac{1}{r^3}\dr = \frac{\pi}{n^2}.
\]
We do the same for the part where $n < \abs{l} < 2n$:
\[
\int_{ \{y \in \R^2: n< \abs{y} < 2n \}} \frac{1}{\abs{y}^2 \abs{k-y}^2}\dy \leq \frac{8 \pi}{n^2} \int_{\frac12}^{3n} \frac{1}{r}\dr \leq \frac{c}{n^2} \log (n).
\]
Now choose $0< \eps < \min \{s-1,1\}$ and estimate $n^{-2} \leq n^{-2\eps} \abs{k}^{2\eps-2}$, hence
\[
\int \snorm{\pi_n (B - B^n)(u)}{-s}^2 \dinvM(u) \leq c \sum_{k \in I_n} \abs{k}^{2-2s} \frac{\log (n)}{n^2} \leq c \sum_{k \in I_n} \abs{k}^{-2s+2\eps} \frac{\log (n)}{n^{2\eps}}.
\]
The sum is convergent as $n \to \infty$ therefore the statement is proven.
\end{proof}
\begin{lem}\label{lem:energyC}
Let $\sigma, \nu > 0$ be arbitrary. Then, $\snorm{C}{-s} \in L^2(\invM)$ if and only if $s > 0$. In particular, there exists a constant $c$ uniform in $n$ such that
\[
\int \snorm{\pi_n (C - C^n)(u)}{-s}^2 \dinvM (u) \leq c n^{-1} \xrightarrow{n\to \infty} 0.
\]
Moreover, $C(u)$ is an element of $H^{-s}$ for $\invM$-a.\,e. u.
\end{lem}
\begin{proof}
The first and last part are stated for a similar presentation to Lemma \ref{lem:energyB}. However, $C$ is only linear in $u$, hence these points are obvious and one can easily verify that $\int \abs{C_k^n(u)}^2 \dinvM(u) \leq c \abs{k}^{-2}$ with a constant $c$ uniform in $n$. This immediately yields the summability in $H^{-s}$ for $s>0$. The convergence rate can be obtained similar to Lemma \ref{lem:energyB}. For $k \in I_n$
\begin{align*}
&\int \abs{C_k(u) - C_k^n(u)}^2 d\invM(u) = \beta^2 \sum_{l \in \zstar \setminus I_n} \big(\gamma_l^k\big)^2 \int u_l^2 \dinvM(u)\\
&\quad= \frac{\beta^2 \sigma^2}{\nu} \sum_{l \in \zstar \setminus I_n} \frac{(k^\bot l)^2}{\abs{k}^2 \abs{l}^4} \frac{1}{(k_2 + l_2)^2} \delta_{k_1+l_1} ( 1-\delta_{k_2+l_2}) = \frac{\beta^2 \sigma^2}{\nu} \frac{k_1^2}{\abs{k}^2} \sum_{j \in \Z, j^2 > n^2-k_1^2} \frac{1}{(k_1^2 + j^2)^2}.
\end{align*}
The sum is again bounded up to a uniform constant by
\[
\int^\infty_{\sqrt{n^2-k_1^2}} \frac{1}{(k_1^2 + y^2)^2}\dy \leq \frac{1}{n^2} \int^\infty_{\sqrt{n^2-k_1^2}} \frac{1}{(k_1^2 + y^2)}\dy \leq \frac{1}{n^2 \abs{k_1}}.
\]
Therefore we have
\[
\int \abs{C_k(u) - C_k^n(u)}^2 d\invM(u) \leq \frac{c}{n^2 \abs{k}} \leq \frac{c}{n \abs{k}^2}
\]
and this proves the lemma.
\end{proof}
\section{Estimates for the Solution of the Resolvent Problem}
In this section we prove integrated gradient estimates for the solution of the resolvent problem $(\lambda - \Kol^n) \psi = \phi$, $\lambda > 0$. Lemma \ref{lem:energyB} suggests that these a priori estimates have to be done in the space $H^{1+s}$, $s>0$. To simplify notation, we introduce the spaces $W^{1,2}_s$ as the closure of $\fcb^1$ in $L^2(\invM)$ w.\,r.\,t. the bilinear form
\[
\mathcal{E}^s(\phi,\psi) \df \sum_{k \in \zstar} \abs{k}^{2s} \int \Dk \phi \Dk \psi \dinvM, \quad \phi, \psi \in \fcb^1
\]
and denote by $\snorm{\cdot}{W^{1,2}_s}$ the corresponding norm. We will use this norm for functions on $H_n$ via the canonical embedding $\iota_n$ without explicit mention. The following proposition is a conclusion of the results in \cite{StannatDirichlet} and yields a first a priori estimate.
\begin{prop}\label{prop:L1UniqueFinite}
The closure $(\overline{K}_{\sigma,\nu}^n, D(\overline{K}_{\sigma,\nu}^n))$ of $(\Kol^n, C_b^2(H_n))$ in $L^1(\invM^n)$ generates a Markovian $C_0$-semigroup of contractions $(\overline{T}_t^n)_{t \geq 0}$. Thus, the operator $(\Kol, C_b^2(H_n))$ is $L^1$-unique. Moreover,
\[
D\big(\overline{K}_{\sigma,\nu}^n\big)_b \df D\big(\overline{K}_{\sigma,\nu}^n\big) \cap L^\infty\big(\invM^n\big) \subset \pi_n\big(W^{1,2}_0\big) 
\]
and
\begin{equation}\label{eq:L2estimate}
\frac{\sigma^2}{2} \sum_{k \in I_n} \int \abs{\Dk \phi}^2 d\invM^n \leq - \int \overline{K}_{\sigma,\nu}^n \phi \phi d\invM^n, \quad \phi \in D\big(\overline{K}_{\sigma,\nu}^n\big)_b.
\end{equation}
\end{prop}
\begin{proof}
\cite[Theorem  I.1.5]{StannatDirichlet} implies the existence of a closed extension on $L^1(\invM^n)$ generating a sub-Markovian semigroup of contractions $(\overline{T}_t^n)_{t \geq 0}$. In particular $D(\overline{K}_{\sigma,\nu}^n)_b \subset D(\mathcal{E}^0|_{H_n}) = \pi_n(W^{1,2}_0)$ and inequality \eqref{eq:L2estimate} holds.

By \cite[Proposition I.1.10]{StannatDirichlet} the measure $\invM^n$ is $(\overline{T}_t^n)_{t \geq 0}$-invariant because $B_k^n$ and $C_k^n \in L^1(\invM^n)$ and therefore $\overline{T}_t^n 1 = 1$ holds. Hence the semigroup is Markovian and \cite[Corollary I.2.2]{StannatDirichlet} implies $L^1$-uniqueness.
\end{proof}
Inequality \eqref{eq:L2estimate} implies the following a priori estimate for the corresponding resolvent $\Res = (\lambda - \overline{K}_{\sigma,\nu}^n)^{-1}$, $\lambda >0$.
\begin{cor}\label{cor:W0estimate}
Let $\psi \in \mathcal{B}_b(H_n)$ and $\lambda >0$. Then $\Res \psi \in D(\overline{K}_{\sigma,\nu}^n)_b$ and
\[
\snorm{\Res \psi}{W^{1,2}_0}^2 \leq \frac{2}{\lambda \sigma^2} \snorm{\psi}{L^\infty}.
\]
\end{cor}
\begin{proof}
Clearly  $\Res \psi \in D(\overline{K}_{\sigma,\nu}^n)$ because $\rg (\Res) \subset D(\overline{K}_{\sigma,\nu}^n)$ and of course $\mathcal{B}_b(H_n) \subset L^1(\invM^n)$. Furthermore, the boundedness follows from the Markovianity of $\lambda \Res$. Therefore, we can use \eqref{eq:L2estimate} and conclude
\begin{align*}
&\frac{\sigma^2}{2} \sum_{k \in I_n} \int \abs{\Dk \Res \psi}^2 d\invM^n + \lambda \int \abs{\Res \psi}^2 d\invM^n\\
&\quad\leq -  \int \overline{K}_{\sigma,\nu}^n \Res\psi \Res\psi d\invM^n + \lambda \int \abs{\Res \psi}^2 d\invM^n\\
&\quad= \int \psi \Res \psi d\invM^n \leq \snorm{\psi}{L^\infty} \snorm{\Res \psi}{L^\infty} \leq \frac{1}{\lambda} \snorm{\psi}{L^\infty}^2.\qedhere
\end{align*}
\end{proof}
However, this integrated gradient estimate is not enough to show $L^1$-uniqueness of the operator $(\Kol, \mathcal{F}C_b^2)$ and we need the following improvement which is the essential part of the proof of Theorem \ref{thm:UniquenessCoriolis}.
\begin{prop}\label{prop:W1+sEstimate}
Suppose Assumption \ref{assum:ViscosityCondition} holds. Let $s \in (0,1]$ and $\psi \in C_b^1(H_n)$. Then there exists $\delta > 0$ and $c(\delta)$ independent of $n$ such that
\[
\snorm{\Res \psi}{W^{1,2}_{1+s}}^2 \leq \frac{1}{4 \delta \lambda} \snorm{\psi}{W^{1,2}_{s}}^2 + c(\delta) \big(\log (n)\big)^{\frac{1+s}{s}} \snorm{\psi}{L^\infty}^2.
\]
\end{prop}
The proof of this proposition will be divided into several technical lemmas. The first one is standard and identifies the commutator of $D$ and $\Kol^n$. 
\begin{lem}\label{lem:Commutator}
Let $s \in \R$ and $\phi \in C_b^3(H_n)$. Then
\begin{align*}
&\sum_{k \in I_n}{\kern -3pt} \abs{k}^{2s} {\kern -3pt}\int{\kern -3pt} \Dk (\Kol^n \phi) \Dk \phi \dinvM^n =  -\frac{\sigma^2}{2} {\kern -3pt}\sum_{k,l \in I_n}{\kern -3pt} \abs{k}^{2s} {\kern -3pt}\int{\kern -3pt} \abs{\Dl \Dk \phi}^2 \dinvM^n -\nu {\kern -2pt}\sum_{k \in I_n}{\kern -3pt} \abs{k}^{2+2s} {\kern -3pt}\int{\kern -3pt} \abs{\Dk \phi}^2 \dinvM^n\\
&\quad - \sum_{k,l \in I_n} \abs{k}^{2 s} (\beta^l_{\pm k \pm l,k} + \beta^l_{k,\pm l \pm k}) \int u_{\pm l \pm k} \Dl \phi \Dk \phi \dinvM^n + \beta \sum_{k,l \in I_n} \abs{k}^{2 s} \gamma^l_k \int \Dl \phi \Dk \phi \dinvM^n
\end{align*}
\end{lem}
\begin{rem}
The useful terms are both negative summands on the right hand side. The second one is exactly the one needed for the gradient estimate in Proposition \ref{prop:W1+sEstimate} and small viscosity $\nu$ results in worse estimates.
\end{rem}
\begin{proof}
Let $k \in I_n$ be fixed, then
\[
\big(\Dk \Kol^n \phi\big)(u) = \big(\Kol^n \Dk \phi\big)(u) - \nu \abs{k}^2 \Dk \phi(u) - \sum_{l \in I_n} \Dk\big( B_l^n(u) + C_l^n(u) \big) \Dl \phi(u).
\]
The derivatives of the Fourier coefficients of $B$ and $C$ can be given explicitly.
\begin{align*}
\Dk B_l^n(u) =& \Dk \Big(\sum_{i,j \in I_n} \beta^l_{i,j} u_i u_j \Big) = \sum_{i \in I_n} \big(\beta^l_{i,k} + \beta^l_{k,i}\big) u_i\\
= &\big(\beta^l_{l-k,k} + \beta^l_{k,l-k}\big) u_{l-k} + \big(\beta^l_{k-l,k} + \beta^l_{k,k-l}\big) u_{k-l} \\
&+ \big(\beta^l_{k+l,k} + \beta^l_{k,k+l}\big) u_{k+l} + \big(\beta^l_{-k-l,k} + \beta^l_{k,-k-l}\big) u_{-k-l}
\end{align*}
as long as all indices are in $I_n$ and $\Dk C_l^n(u) = -\beta \gamma^l_k$. The identity $\Kol^n (\psi^2) = 2 \psi \Kol^n(\psi) + \sigma^2 \sum_{l \in I_n} \abs{\Dl \psi}^2$ for all $\psi \in C_b^2(H_n)$ together with the invariance of $\invM^n$ implies
\[
\sum_{k \in I_n} \abs{k}^{2s}\int \Kol^n(\Dk \phi) \Dk \phi \dinvM^n = - \frac{\sigma^2}{2} \sum_{k,l \in I_n} \abs{k}^{2s} \int \abs{\Dl \Dk \phi}^2 \dinvM^n.
\]
Consequently, for any $s \in \R$
\begin{align*}
&\sum_{k \in I_n} \abs{k}^{2s} \int \Dk (\Kol^n \phi) \Dk \phi \dinvM^n =  -\frac{\sigma^2}{2} \sum_{k,l \in I_n} \abs{k}^{2s} \int \abs{\Dl \Dk \phi}^2 \dinvM^n\\
&\qquad - \nu \sum_{k \in I_n} \abs{k}^{2+2s} \int \abs{\Dk \phi}^2 \dinvM^n - \sum_{k,l \in I_n} \abs{k}^{2s} \int ( \Dk B_l^n(u) + \Dk C_l^n(u) ) \Dl \phi \Dk \phi \dinvM^n.\qedhere
\end{align*}
\end{proof}
In the course of the proof of Proposition \ref{prop:W1+sEstimate} we will replace $\phi$ by the resolvent. In particular, the additional commutator terms have to be estimated in terms of the two negative ones. Because of its linear structure, the Coriolis term is easier to handle and we get an estimate independent of $\sigma$ and $\nu$. One key tool is the following. Let $\phi \in \fcb^1$ and $0 < s_0 < 1 + s$. Then, for any $\delta > 0$
\begin{equation}\label{eq:SobolevInterpolation}
\log (n) \snorm{D \phi}{s_0}^2 \leq \delta \snorm{D\phi}{1+s}^2 + c(s, s_0, \delta) \big(\log (n)\big)^{\frac{1+s}{1+s-s_0}} \snorm{D\phi}{0}^2.
\end{equation}
This relation follows from the interpolation inequality \eqref{eq:Interpolation} applied pointwise for fixed $u$ and Young's inequality.
\begin{lem}\label{lem:CommCoriolis}
Let $s \in (0,1]$ and $\phi \in C_b^1(H_n)$. Then, for every $\delta > 0$ there exists $c(\delta) < \infty$ independent of $n$ such that
\[
\beta \sum_{k,l \in I_n} \abs{k}^{2s} \gamma_k^l \int \Dl \phi \Dk \phi \dinvM^n \leq \delta \sum_{k \in I_n} \abs{k}^{2 + 2s} \int \abs{\Dk \phi}^2 \dinvM^n + c(\delta) \sum_{k \in I_n} \int \abs{\Dk \phi}^2 \dinvM^n.
\]
\end{lem}
\begin{proof}
Obviously, it holds that
\[
\beta \sum_{k,l \in I_n} \abs{k}^{2s} \gamma_k^l \int \Dl \phi \Dk \phi \dinvM^n \leq \beta \sum_{k,l \in I_n} \frac{\abs{k}^{2s}}{\abs{k+l}} \delta_{k_1 + l_1} (1 - \delta_{k_2 + l_2})  \int \abs{\Dl \phi} \abs{\Dk \phi} \dinvM^n.
\]
Note that the constraints on the indices imply $\abs{k+l} \neq 0$ and also yield a summation over only a one dimensional subset of $\zstar$. With $\abs{k}^s \leq \abs{k+l}^s + \abs{l}^s$ for $s \in (0,1]$ follows
\[
\frac{\abs{k}^{2s}}{\abs{k+l}} \leq \frac{\abs{k}^s}{\abs{k+l}^{1-s}} + \frac{\abs{k}^s \abs{l}^s}{\abs{k+l}}
\]
and Young's inequality with $p=q=2$ implies
\begin{align*}
\frac{\abs{k}^{2s}}{\abs{k+l}} \abs{\Dl \phi} \abs{\Dk \phi} &\leq \frac12 \frac{\abs{k}^{\frac32 + 2s}}{\abs{k+l}^{1-s} \abs{l}^{\frac32}} \abs{\Dk \phi}^2 + \frac12 \frac{\abs{l}^{\frac32}}{\abs{k+l}^{1-s} \abs{k}^{\frac32}} \abs{\Dl \phi}^2\\
&\quad+ \frac12 \frac{\abs{k}^{1 + 2s}}{\abs{k+l} \abs{l}} \abs{\Dk \phi}^2 + \frac12 \frac{\abs{l}^{1+2s}}{\abs{k+l}\abs{k}}\abs{\Dl \phi}^2.
\end{align*}
For fixed $k$ and $l$, all denominators are summable in $l$ and $k$, respectively. Thus, we just derived
\[
\sum_{k,l \in I_n} \frac{\abs{k}^{2s}}{\abs{k+l}} \delta_{k_1 + l_1} (1 - \delta_{k_2 + l_2}) \abs{\Dl \phi} \abs{\Dk \phi}\leq c \Bigl( \snorm{D\tphi}{\frac34 + s}^2 + \snorm{D\phi}{\frac34}^2 + \snorm{D\tphi}{\frac12 + s}^2 \Bigr),
\]
where $\tphi \in \fcb^1$ is the extension of $\phi$ via $\iota_n$ to $\fcb^1$. An application of the interpolation inequality \eqref{eq:SobolevInterpolation} yields the desired result.
\end{proof}
Of course, we want to achieve a similar result for the convection term. The critical step is the integration by part formula in Lemma \ref{lem:PartialIntegrationGaussian}. This eliminates $u_{\pm k \pm l}$ but yields a second derivative of $\phi$.
\begin{lem}\label{lem:CommConvection}
Let $s \in (0,1]$ and $\phi \in C_b^2(H_n)$. Then, for every $\eps, \delta > 0$ there exists $c(\eps,\delta)< \infty$ independent of $n$ such that
\begin{align*}
&\sum_{k,l \in I_n} \abs{k}^{2 s} (\beta^l_{\pm k \pm l,k} + \beta^l_{k,\pm l \pm k}) \int u_{\pm l \pm k} \Dl \phi \Dk \phi \dinvM^n \leq 4 \eps \sigma^2 \sum_{k,l \in I_n} \abs{k}^{2s} \int \abs{\Dl \Dk \phi} \dinvM^n\\
&\qquad +  c(\eps, \delta) \big(\log (n)\big)^{\frac{1+s}{s}} \sum_{k \in I_n} \int \abs{\Dk \phi}^2 \dinvM^n + \frac{\sigma^2}{\nu^2}\frac{5(\delta + S(2))}{ \pi^2 \eps} \int \sum_{k \in I_n} \abs{k}^{2+2s} \abs{\Dk \phi}^2 \dinvM^n.
\end{align*}
\end{lem}
\begin{proof}
Recall Lemma \ref{lem:PartialIntegrationGaussian} and apply this to the convection part of the commutator which yields
\[
\int u_{\pm k \pm l} \Dl \phi \Dk \phi \dinvM^n(u) = \frac{\sigma^2}{\nu \abs{ k \pm l}^2} \int \bigl(\partial_{\pm k \pm l} \Dl \phi\bigr) \Dk \phi + \Dl \phi \bigl(\partial_{\pm k \pm l} \Dk \phi \bigr) \dinvM^n.
\]
As a next step we need an estimate for $\abs{ \beta^l_{\pm l \pm k,k} + \beta^l_{k,\pm l \pm k}}$, namely
\[
\frac{\sqrt{2}}{4 \pi} \Biggl\lvert \frac{\pm (l^\bot\cdot (l \pm k)) ( l \cdot k)}{\abs{l}\abs{l \pm k}\abs{k}} \delta_{l, \pm l \pm k, k}+ \frac{\pm (l^\bot\cdot k) ( l \cdot (l \pm k))}{\abs{l}\abs{l \pm k}\abs{k}} \delta_{l,k, \pm l \pm k}\Biggr\rvert \leq \frac{\sqrt{2}}{2 \pi} \abs{l}.
\]
Combining the last two estimates yields the following upper bound,
\begin{align*}
&\sum_{k,l \in I_n} \abs{k}^{2 s} (\beta^l_{\pm k \pm l,k} + \beta^l_{k,\pm l \pm k}) \int u_{\pm l \pm k} \Dl \phi \Dk \phi \dinvM^n\\
&\quad\leq \frac{\sigma^2}{\nu} \frac{\sqrt{2}}{2\pi} \sum_{k,l \in I_n} \abs{k}^{2 s} \frac{\abs{l}}{\abs{\pm k \pm l}^2}\int \bigl( \abs{\partial_{\pm k \pm l} \Dl \phi} \abs{\Dk \phi} + \abs{\Dl \phi} \abs{ \partial_{\pm k \pm l} \Dk \phi}\bigr)d\invM^n.
\end{align*}
The main task is to control all second derivatives such that they vanish in the final estimate. This is similar to the proof of Lemma \ref{lem:CommCoriolis} for the Coriolis part and for a shorter notation we consider only the case $\pm k \pm l = l-k$ in the following. The other three cases are done in the same way. We also have to remark, that we cannot use the estimates in \cite{Stannat2DNSE} to derive the result, since our modified proof involves the a priori estimate from Corollary \ref{cor:W0estimate} and some logarithmic growth in $n$. It is matched by the sharp convergence results in Lemmas \ref{lem:energyB} and \ref{lem:energyC}.

Essentially, we have to take care of two terms. The first one is estimated as follows, using $\abs{k}^s \leq \abs{l}^s + \abs{l-k}^s$ for $s \in (0,1]$.
\[
\frac{\sigma^2}{\nu} \frac{\sqrt{2}}{2 \pi} \sum_{k,l \in I_n} \frac{\abs{k}^{2s} \abs{l}}{\abs{l-k}^2} \abs{\partial_{l-k} \Dl \phi} \abs{\Dk \phi} \leq \frac{\sigma^2}{\nu} \frac{\sqrt{2}}{2\pi} \sum_{k,l \in I_n} \frac{\abs{k}^s \abs{l}}{\abs{l-k}^2} \bigl( \abs{l}^s + \abs{l-k}^s \bigr) \abs{\partial_{l-k} \Dl \phi} \abs{\Dk \phi}
\]
An application of Young's inequality with $p=q=2$ and a coefficient $\eps > 0$ together with $\abs{l}^2 \leq 2\abs{k}^2 + 2 \abs{l-k}^2$ yields
\begin{align*}
&\leq \frac{\eps}{2} \sigma^2 \sum_{k,l \in I_n}\abs{l}^{2s} \abs{\partial_{l-k} \Dl \phi}^2 + \frac{\eps}{2} \sigma^2 \sum_{k,l \in I_n}\abs{l-k}^{2s} \abs{\partial_{l-k} \Dl \phi}^2 + \frac{\sigma^2}{\nu^2} \frac{1}{2 \pi^2 \eps} \sum_{k,l \in I_n} \frac{\abs{k}^{2 s}\abs{l}^2 }{\abs{l-k}^4} \abs{ \Dk \phi}^2\\
&\leq \eps \sigma^2 \sum_{k,l \in I_n}\abs{l}^{2s} \abs{\partial_{l-k} \Dl \phi}^2 + \frac{\sigma^2}{\nu^2} \frac{1}{\pi^2 \eps} \sum_{k,l \in I_n} \Bigl( \abs{k}^{2+ 2 s}\frac{1}{\abs{l-k}^4} + \abs{k}^{2 s} \frac{1}{\abs{l-k}^2} \Bigr) \abs{ \Dk \phi}^2.
\end{align*}
Clearly $\abs{l-k}^{-4}$ is summable over $l \in \zstar$. It follows
\[
\sum_{k,l \in I_n} \abs{k}^{2+ 2 s}\frac{1}{\abs{l-k}^4} \abs{ \Dk \phi}^2 \leq S(2) \sum_{k \in I_n} \abs{k}^{2+ 2s} \abs{ \Dk \phi}^2 = S(2) \snorm{D \tphi}{1+s}.
\]
Again, denote by $\tphi$ the extension of $\phi$ to $\fcb^2$ via $\iota_n$. Similarly,
\[
\sum_{k,l \in I_n} \abs{k}^{2 s} \frac{1}{\abs{l-k}^2} \abs{\Dk \phi}^2 \leq c \log (n) \snorm{D \tilde{\phi}}{s}.
\]
The logarithmic growth in $n$ is sufficiently small and we use, as in Lemma \ref{lem:CommCoriolis}, the interpolation inequality \eqref{eq:SobolevInterpolation} to obtain
\[
\sum_{k,l \in I_n} \abs{k}^{2 s} \frac{1}{\abs{l-k}^2} \abs{\Dk \phi}^2 \leq \delta \snorm{D \tilde{\phi}}{1+s}^2 + c(s,\delta) \big(\log (n)\big)^{(1+s)} \snorm{D \tilde{\phi}}{0}^2.
\]
Essentially, we just found the estimate
\begin{equation}\label{proof:estimate1}
\begin{split}
&\frac{\sigma^2}{\nu} \frac{\sqrt{2}}{2 \pi} \sum_{k,l \in I_n} \frac{\abs{k}^{2s} \abs{l}}{\abs{l-k}^2} \abs{\partial_{l-k} \Dl \phi} \abs{\Dk \phi} \leq \eps \sigma^2 \sum_{k,l \in I_n}\abs{l}^{2s} \abs{\partial_{l-k} \Dl \phi}^2\\
&\qquad+ \frac{\sigma^2}{\nu^2}\frac{\delta + S(2)}{\pi^2 \eps} \sum_{k \in I_n} \abs{k}^{2+2s} \abs{\Dk \phi}^2 + \frac{\sigma^2}{\nu^2}\frac{c(s, \delta)}{\pi^2 \eps} \big(\log (n)\big)^{(1+s)} \sum_{k \in I_n} \abs{\Dk \phi}^2.
\end{split}
\end{equation}
As the next step, we have to estimate the remaining terms in similar ways. Note that the roles of $k$ and $l$ are not symmetric, thus the estimates differ. With $\abs{k}^{2s} \leq 2 \abs{l-k}^{2s} + 2 \abs{l}^{2s}$ it follows that
\begin{align*}
&\frac{\sigma^2}{\nu} \frac{\sqrt{2}}{2 \pi} \sum_{k,l \in I_n} \frac{\abs{k}^{2s} \abs{l}}{\abs{l-k}^2} \abs{\Dl \phi}\abs{\partial_{l-k} \Dk \phi}\\
&\quad\leq \eps \sigma^2 \sum_{k,l \in I_n}\abs{k}^{2s} \abs{\partial_{l-k} \Dk \phi}^2 + \frac{\sigma^2}{\nu^2} \frac{1}{8 \pi^2 \eps} \sum_{k,l \in I_n} \frac{\abs{k}^{2 s}\abs{l}^2 }{\abs{l-k}^4} \abs{ \Dl \phi}^2\\
&\quad\leq \eps \sigma^2 \sum_{k,l \in I_n} \abs{k}^{2s} \abs{\partial_{l-k} \Dk \phi}^2 + \frac{\sigma^2}{\nu^2} \frac{1}{4 \pi^2 \eps}\sum_{k,l \in I_n} \left( \frac{\abs{l}^2}{\abs{l-k}^{4-2s}} + \frac{\abs{l}^{2+2s} }{\abs{l-k}^4}\right) \abs{ \Dl \phi}^2
\end{align*}
Again, $\abs{l-k}^{-4}$ is summable in $k$, so it is exactly treated like above. $\abs{l-k}^{-4+2s}$ is summable in $k$ if $s < 1$ and of order $\log n$ if $s=1$. So we bound it similarly to the case above by
\[
\sum_{k,l \in I_n} \frac{\abs{l}^2}{\abs{l-k}^{4-2s}} \abs{ \Dl \phi}^2 \leq c \log (n) \snorm{D \tilde{\phi}}{1} \leq \delta \snorm{D \tilde{\phi}}{1+s}^2 + c(s,\delta) \big(\log (n)\big)^{\frac{1 + s}{s}} \snorm{D \tilde{\phi}}{0}^2.
\]
Thus, we arrive at an estimate as in \eqref{proof:estimate1}. It is clear, that the other three cases of $\pm k \pm l$ can be estimated in the exact same way.
\end{proof}
In our next step we derive an overall estimate by combining Lemmas \ref{lem:Commutator}, \ref{lem:CommCoriolis} and \ref{lem:CommConvection}. Note that the second derivatives in this equation would not appear if the fluid was not perturbed by a random noise -- this can be interpreted as a regularizing effect of the noise.
\begin{lem}\label{lem:OverallEstimate}
Let $s \in (0,1]$, $\phi \in C_b^3(H_n)$ and $\lambda > 0$. Then, there exists $\delta > 0$ such that
\begin{equation}\label{eq:OverallEstimate}
\begin{split}
&\lambda \sum_{k \in I_n} \abs{k}^{2s} \int \abs{\Dk \phi}^2 \dinvM^n + \delta \sum_{k \in I_n} \abs{k}^{2+2s} \int \abs{\Dk \phi}^2 \dinvM^n \\
&\quad\leq \sum_{k \in I_n} \abs{k}^{2s} \int \Dk ((\lambda - \Kol^n) \phi) \Dk \phi \dinvM^n + c(\delta) \big(\log (n)\big)^{\frac{1+s}{s}} \sum_{k \in I_n} \int \abs{ \Dk \phi}^2 \dinvM^n,
\end{split}
\end{equation}
where $c(\delta) < \infty$ independent of $n$.
\end{lem}
\begin{rem}
This lemma is an improvement in comparison to \cite{Stannat2DNSE}. We weaken the smallness condition for $\nu$ by trading this to some growth in $n \in \N$ in front of a weaker norm. This is sufficiently small to be matched by the convergence of $B^n$ and $C^n$ later on. Also, note that the parameters $\omega, \beta$ of the Coriolis force do not appear in the smallness condition.
\end{rem}
\begin{proof}
In the preceding lemmas we deduced
\begin{align*}
&\sum_{k \in I_n} \abs{k}^{2s} \int \Dk (\Kol^n \phi) \Dk \phi \dinvM^n \leq  \Bigl(-\frac{\sigma^2}{2} + 4 \eps \sigma^2 \Bigr) \sum_{k,l \in I_n} \abs{k}^{2s} \int \abs{\Dl \Dk \phi}^2 \dinvM^n\\
&+ \Bigl(-\nu + \frac{\sigma^2}{\nu^2} \frac{6 \delta + 5 S(2)}{\pi^2 \eps} \Bigr) \sum_{k \in I_n} \abs{k}^{2+2s} \int \abs{\Dk \phi}^2 \dinvM^n + c(\eps, \delta) \big(\log (n)\big)^{\frac{1+s}{s}} \sum_{k \in I_n} \int \abs{\Dk \phi}^2 \dinvM^n
\end{align*}
Choosing $\eps = \frac{1}{8}$ will provide that all the second derivatives of $\phi$ vanish. Now by Assumption \ref{assum:ViscosityCondition}, set $\delta \df \nu^3 \pi^2/(7 \sigma^2) - (40/7) S(2) > 0$ and the assertion follows immediately.
\end{proof}
\begin{lem}\label{lem:Extension}
Inequality \eqref{eq:OverallEstimate} extends to all $\phi = \Res \psi$ with $\psi \in C_b^1(H_n)$.
\end{lem}
\begin{rem}
The proof of this lemma follows \cite[Lemma 2.6]{Stannat2DNSE}, which appears to be slightly inaccurate since there the identity (20) does not hold. However, the remaining proof can be modified, as done below. In particular, the statement in \cite[Lemma 2.6]{Stannat2DNSE} is also valid.
\end{rem}
\begin{proof}
In a first step, we need a different uniqueness result for $\Kol^n$, in particular \cite[Theorem 2.5, Chapter 2.F]{Eberle}. The statement says that $(\Kol^n, C_0^\infty(H_n) )$ is $L^2$-unique, hence $C_0^\infty(H_n)$ is a core for $\Kol^n$, i.\,e. dense w.\,r.\,t. the graph norm. This implies that for fixed $\psi \in C_b^1(H_n)$, we can find a sequence $(\phi_m) \subset C_0^\infty(H_n)$ such that
\[
\lim_{m \to \infty} \Bigl( \snorm{\phi_m - \Res \psi}{L^2(H_n, \invM^n)} + \snorm{\Kol^n \phi_m - \Kol^n \Res \psi}{L^2(H_n, \invM^n)} \Bigr) = 0.
\]
Now consider
\[
L_s^n \phi (u) \df \frac{\sigma^2}{2} \sum_{k \in I_n} \abs{k}^{2s} \Dk^2 \phi (u) - \nu \sum_{k \in I_n} \abs{k}^{2+2s} u_k \Dk \phi(u),
\]
which is the generator associated to the bilinear form $\mathcal{E}^s$, i.\,e.
\[
\int L_s^n \phi \phi \dinvM^n = - \mathcal{E}^s(\phi,\phi),
\]
since with the integration by parts from Lemma \ref{lem:PartialIntegrationGaussian} it follows that
\[
\frac{\sigma^2}{2} \sum_{k \in I_n} \abs{k}^{2s} \int \Dk^2 \phi \phi \dinvM^n = -\frac{\sigma^2}{2} \sum_{k \in I_n} \abs{k}^{2s} \int \Dk \phi \Bigl( \Dk \phi - \frac{2\nu \abs{k}^2}{\sigma^2} \phi \Bigr) \dinvM^n.
\]
The bilinear form is used in the gradient estimates in Lemma \ref{lem:OverallEstimate} and in the following we want to prove that $L_s^n \phi_m \to L_s^n \Res \psi$ weakly along some subsequence. For this purpose consider
\begin{align}
\int \bigl( L_s^n \phi \bigr)^2 \dinvM^n &= - \mathcal{E}^s (L_s^n \phi, \phi)\notag\\
&= -\frac{\sigma^2}{2} \sum_{k \in I_n} \abs{k}^{2s} \int L_s^n \Dk \phi \Dk \phi \dinvM^n + \frac{\sigma^2 \nu}{2} \sum_{k \in I_n} \abs{k}^{2 + 4s} \int \abs{\Dk \phi}^2 \dinvM^n\notag\\
&\leq \frac{\sigma^4}{4} \sum_{k,l \in I_n} \abs{k}^{2s} \abs{l}^{2s} \int \abs{\Dl \Dk \phi}^2 \dinvM^n + c(n) \mathcal{E}^s(\phi,\phi),\label{proof:UpperLowerBound}
\end{align}
with some constant $c(n)$. A trivial lower bound is given by the above with $c(n) = 0$. These inequalities immediately imply that we can switch between different values of $s$, because
\begin{equation}\label{proof:Different_s}
\int \bigl( L_{s_1}^n \phi \bigr)^2 \dinvM^n \leq c(s_1, s_2, n) \int \bigl( L_{s_2}^n \phi \bigr)^2 \dinvM^n.
\end{equation}
Thus, we set $s= 0$ in the following and the proof of Lemma \ref{lem:OverallEstimate} with the choice $\eps = \frac{1}{16}$ instead of $\frac{1}{8}$ yields
\[
\frac{\sigma^4}{4} \sum_{k,l \in I_n} \int \abs{\Dl \Dk \phi}^2 \dinvM^n \leq - \sigma^2 \mathcal{E}^0 (\Kol^n \phi, \phi) + C(n) \mathcal{E}^0(\phi,\phi).
\]
Now we use the lower bound obtained in \eqref{proof:UpperLowerBound}, \eqref{eq:L2estimate} and the fact that $L_0^n$ is associated to $\mathcal{E}^0$, together with H\"older's and Young's inequality.
\begin{align}
\int \bigl( L_0^n \phi \bigr)^2 \dinvM^n &\leq 2 \int L_0^n \phi \Kol^n \phi \dinvM^n - c(n) \int \Kol^n \phi \phi \dinvM^n\notag\\
&\leq \frac12 \int \bigl( L_0^n \phi \bigr)^2 \dinvM^n + 2 \snorm{\Kol^n \phi}{L^2}^2 + c(n) \snorm{\Kol^n \phi}{L^2} \snorm{\phi}{L^2},\notag\\
\intertext{hence}
\int \bigl( L_0^n \phi \bigr)^2 \dinvM^n &\leq 4 \snorm{\Kol^n \phi}{L^2}^2 + c(n) \snorm{\Kol^n \phi}{L^2} \snorm{\phi}{L^2}.\label{proof:EstimateL0}
\end{align}
Recall \eqref{proof:Different_s} and we deduce
\begin{align*}
&\Bigl| \frac{\sigma^2}{2} \sum_{k \in I_n} \abs{k}^{2s} \int \Dk ((\lambda - \Kol^n)\phi) \Dk \phi \dinvM \Bigr| \\
&\quad= \abs{\mathcal{E}^s((\lambda - \Kol^n)\phi, \phi)} = \Bigl| \int (\lambda - \Kol^n)\phi L_s^n \phi \dinvM \Bigr|\\
&\quad\leq \snorm{ (\lambda - \Kol^n)\phi}{L^2}^2 + \snorm{L_s^n \phi}{L^2}^2 \leq \snorm{ (\lambda - \Kol^n)\phi}{L^2}^2 + c(s, n) \snorm{L_0^n \phi}{L^2}^2\phantom{\Big|}\\
&\quad\leq \snorm{ (\lambda - \Kol^n)\phi}{L^2}^2 + c(s, n) \bigl(\snorm{\Kol^n \phi}{L^2}^2 + c(n) \snorm{\Kol^n \phi}{L^2} \snorm{\phi}{L^2}\bigr).\phantom{\Big|}
\end{align*}
Now we turn back to the sequence $(\phi_m)$ and due to \eqref{proof:EstimateL0} we know that
\[
\sup_m \int \bigl( L_0^n \phi_m \bigr)^2 \dinvM^n < \infty,
\]
hence $\Res \psi \in D(L_0^n) = D(L_s^n)$ for all $s \in \R$. In particular, the boundedness implies weak convergence in $L^2(\invM^n)$ of $L_s^n \phi_{m_l} \to L_s^n \Res \psi$ along some subsequence $(m_l)$, thus
\[
\Bigl| \frac{\sigma^2}{2} \sum_{k \in I_n} \abs{k}^{2s} \int \Dk ((\lambda - \Kol^n)\phi_{m_l}) \Dk \phi_{m_l}  - \Dk \psi \Dk \Res \psi \dinvM \Bigr| \to 0.
\]
Inequality \eqref{eq:OverallEstimate} holds for all $\phi_{m_l}$ and the assertion follows by Lebesgue's dominated convergence theorem.
\end{proof}
\begin{proof}[Proof of Proposition \ref{prop:W1+sEstimate}]
The rest of the proof is a simple manipulation. We have shown that for $\psi \in C_b^1(H_n)$
\begin{align*}
&\lambda \sum_{k \in I_n} \abs{k}^{2 s} \int \abs{\Dk \Res\psi}^2 \dinvM^n + \delta \sum_{k \in I_n} \abs{k}^{2+2 s} \int \abs{\Dk \Res\psi}^2 \dinvM^n \\
&\leq \sum_{k \in I_n} \abs{k}^{2 s} \int \Dk \psi \Dk \Res \psi \dinvM^n + c(\delta) \big(\log (n)\big)^{\frac{1+s}{s}} \sum_{k \in I_n} \int \abs{ \Dk \Res \psi}^2 \dinvM^n\\
&\leq \frac{1}{4 \lambda} {\kern -2pt}\sum_{k \in I_n} {\kern -2pt}\abs{k}^{2 s} {\kern -2pt}\int {\kern -2pt}\abs{\Dk \psi}^2 \dinvM^n + \lambda {\kern -2pt}\sum_{k \in I_n} {\kern -2pt}\abs{k}^{2 s} {\kern -2pt}\int {\kern -2pt}\abs{\Dk \Res\psi}^2 \dinvM^n + \frac{2 c(\delta)}{\lambda \sigma^2} \big(\log (n)\big)^{\frac{1+s}{s}}\snorm{\psi}{L^\infty}^2.
\end{align*}
Rearranging the terms yields the result.
\end{proof}
\section{Proof of Theorem \ref{thm:UniquenessCoriolis}}
The remaining part of the proof is fairly standard. By general arguments $(\Kol, \fcb^2)$ is dissipative, hence closable in $L^1(\invM)$. Thus, it remains to check the range condition $(\lambda - \Kol) (\fcb^2) \subset L^1(\invM)$ dense for some $\lambda >0$, see e.\,g. \cite{Eberle}.

Let us fix a function $\psi \in C_b^1(H_{n_0})$ for some finite $n_0$. Clearly, $\psi$ has its representative $\tilde{\psi} \in \fcb^1$ and can be considered as a function on $H_n$ for arbitrary $n \geq n_0$. Thus, the resolvent $\Res \psi \in D(\overline{K}_{\sigma,\nu})_b$ for all $n$ and
\begin{align*}
(\lambda - \overline{K}_{\sigma, \nu}) \Res \psi &= (\lambda - \overline{K}^n_{\sigma, \nu}) \Res \psi + (\overline{K}^n_{\sigma, \nu} - \overline{K}_{\sigma, \nu}) \Res \psi \\
&= \psi + \sum_{k \in I_n} \bigl(B^n_k - B_k + C^n_k - C_k \bigr)\Dk \Res \psi.
\end{align*}
Now, we combine the convergence of the Galerkin approximations with the integrated gradient estimated for the resolvent. Let $s \in (0,1]$, then for any $0 < \eps < s$
\begin{align*}
&\snorm{ (\lambda - \overline{K}_{\sigma, \nu}) \Res \psi - \psi}{L^1}\\
&\quad\leq \snorm{\Res \psi}{W^{1,2}_{1 + s}} \cdot \Biggl(\int  \snorm{\pi_n (B-B^n)}{-1-s}^2 \dinvM + \int \snorm{\pi_n (C-C^n)}{-1-s}^2 \dinvM \Biggr)^{\frac12}\\
&\quad\leq c(\psi) (1 + \log n)^{\frac{1+s}{s} + \frac12} \cdot c n^{-\eps} \xrightarrow{n \to \infty} 0,
\end{align*}
which implies the denseness of the range $(\lambda - \Kol) \big(\fcb^2\big) \subset L^1(\invM)$, since $\fcb^1 \subset  L^1(\invM)$ dense.\qed
\section*{Acknowledgement}
During the prepatation of this article the author was supported by the DFG and JSPS as a member of the International Research Training Group Darmstadt-Tokyo IRTG 1529.

\end{document}